\documentclass[11pt,hyp,]{nyjm}
\usepackage{hyperref}
\hypersetup{nesting=true,debug=true,naturalnames=true}
\usepackage{graphicx,amssymb,upref}

\let\<\langle
\let\>\rangle

\let\uml\"

\usepackage{fancyhdr}
\usepackage{stmaryrd}
\usepackage{calrsfs}

\newcommand{\N}{\mathbb{N}}

\usepackage{amsmath}
\usepackage[nospace,noadjust]{cite}
\usepackage{amsfonts}
\usepackage{amssymb,enumerate}
\usepackage{amsthm}
\usepackage{cite}
\usepackage{comment}
\usepackage{color}
\usepackage[all]{xy}
\usepackage{hyperref}
\usepackage{lineno}
\usepackage{stmaryrd}

\usepackage{mathtools}
\usepackage{bm}
\usepackage{graphicx}
\usepackage{psfrag}
\usepackage{url}

\usepackage{fancyhdr}
\usepackage{tikz-cd}

\usepackage{dirtytalk}
\usepackage{float}
\usepackage{mathrsfs}

\usepackage[shortlabels]{enumitem}

\newtheorem*{maintheorem*}{Main Theorem}
\newtheorem{theorem}{Theorem}[section]
\newtheorem{prop}[theorem]{Proposition}

\newtheorem{conj}[theorem]{Conjecture}
\newtheorem{lemma}[theorem]{Lemma}
\newtheorem{cor}[theorem]{Corollary}

\theoremstyle{definition}
\newtheorem{definition}[theorem]{Definition}

\newtheorem{example}[theorem]{Example}
\numberwithin{equation}{section}

\newcommand{\nn}{\mathbb{N}}

\newcommand{\zz}{\mathbb{Z}}

\newcommand{\supp}{\mathsf{supp}}



\begin{document}
	
	\mbox{}
	\title{A Goldbach theorem for Laurent series semidomains}

	\author{Nathan Kaplan}
	\address{Department of Mathematics\\University of California \\Irvine, CA 92697}
	\email{nckaplan@math.uci.edu}
	
	\author{Harold Polo}
	\address{Department of Mathematics\\University of California \\Irvine, CA 92697}
	\email{harold.polo@uci.edu}

\date{\today}

\begin{abstract}
	A semidomain is a subsemiring of an integral domain. One can think of a semidomain as an integral domain in which additive inverses are no longer required. A semidomain $S$ is additively reduced if $0$ is the only invertible element of the monoid $(S,+)$, while $S$ is additively atomic if the monoid $(S,+)$ is atomic (i.e., every non-invertible element $s \in S$ can be written as the sum of finitely many irreducibles of $(S,+)$). In this paper, we describe the additively reduced and additively atomic semidomains $S$ for which every Laurent series $f \in S\llbracket x^{\pm 1} \rrbracket$ that is not a monomial can be written as the sum of at most three multiplicative irreducibles. In particular, we show that, for each $k \in \mathbb{N}$, every polynomial $f \in \mathbb{N}[x_1^{\pm 1}, \ldots, x_k^{\pm 1}]$ that is not a monomial can be written as the sum of two multiplicative irreducibles provided that $f(1, \ldots, 1) > 3$.
\end{abstract}

\maketitle

\setcounter{tocdepth}{1}
\tableofcontents

\section{Introduction}
\label{sec:intro}

The Goldbach conjecture, proposed by Goldbach in a letter to Euler in 1742, remains one of the oldest and most famous unsolved problems in mathematics. In modern terms, the conjecture states that every even integer greater than $2$ can be written as the sum of two prime numbers.

Since the 1960s, analogues of the Goldbach conjecture for classes of polynomial rings have been investigated by several authors. In 1965, Hayes~\cite{hayes} showed that polynomials $f \in \mathbb{Z}[x]$ with $\deg(f) > 1$ can be represented as the sum of two irreducibles. His work was later extended by Saidak~\cite{saidak}, Kozek~\cite{kozek}, and Pollack~\cite{pollack}. On the other hand, variations of the Goldbach conjecture for polynomials over finite fields have been studied by Effinger and Hayes~\cite{effinger}, Bender~\cite{bender}, Car and Gallardo~\cite{car}, and others. More recently, Paran~\cite{paran} examined a version of the Goldbach conjecture in the ring of formal power series over the integers.

Liao and the second author~\cite{liao} initiated the study of analogues of the Goldbach conjecture for polynomial semirings. Specifically, they showed that every $f \in \mathbb{N}[x^{\pm 1}]$ that is not a monomial can be written as the sum of two multiplicative irreducibles provided that $f(1) > 3$ \cite[Theorem~1]{liao}. It is worth mentioning that the proof uses a fact about the distribution of prime numbers in $\mathbb{Z}$. Note that this theorem is closely connected to the Goldbach conjecture. This result is about partitioning a fixed set of units into two subsets, where each subset represents an irreducible Laurent polynomial with positive integer coefficients, just as the Goldbach conjecture concerns the partitioning of a fixed set of units into two subsets, where each subset represents a positive prime. Moreover, if the Goldbach conjecture were true then, using the fact that $f \in \mathbb{N}[x^{\pm 1}]$ is a multiplicative irreducible when $f(1)$ is a prime number, one could prove that Laurent polynomials $f \in \mathbb{N}[x^{\pm 1}]$ for which $f(1)$ is an even number greater than $2$ can be written as the sum of two irreducibles.  In this paper, we provide a generalization of \cite[Theorem~1]{liao}, from which we recover the cited result without relying on specific information about prime numbers.

There are polynomial semirings for which no analogue of the Goldbach conjecture holds. For instance, consider the semiring $\mathbb{N}[\sqrt{2}][x^{\pm 1}]$. It is not difficult to show that if $\sqrt{2} = a + b$ for some $a, b \in \mathbb{N}[\sqrt{2}]$ then either $a = 0$ or $b = 0$. As a consequence, we obtain that for any $n \in \N$, the polynomial $\sqrt{2} x^n+\sqrt{2} x^{n-1}+\cdots +\sqrt{2} x + \sqrt{2}$ cannot be written as the sum of $n$ or fewer multiplicative irreducibles. A major goal of this article is to clarify the differences between semirings such as $\mathbb{N}[x^{\pm 1}]$ and $\mathbb{N}[\sqrt{2}][x^{\pm 1}]$ and explain how these differences are connected to analogues of the Goldbach conjecture for classes of Laurent polynomials with coefficients in a semiring.

Following~\cite{polo}, we say that a semidomain is a subsemiring of an integral domain. One can think of a semidomain as an integral domain in which additive inverses are no longer required. Clearly, the class of semidomains contains the class of integral domains. Classes of polynomial semidomains have been investigated before in the literature (see, for instance, \cite{chapmanpolo,CF19,goel}). Note that a semidomain $S$ consists of two monoids, namely $(S\setminus\{0\}, \cdot)$ and $(S,+)$. A semidomain $S$ is additively reduced if $0$ is the only invertible element of the monoid $(S,+)$, while $S$ is additively atomic if the monoid $(S,+)$ is atomic (i.e., every non-invertible element $s \in S$ can be written as the sum of finitely many irreducibles of $(S,+)$). It is easy to see that the set of nonnegative integers is an additively reduced and additively atomic semidomain.

This paper is structured as follows. In Section~\ref{sec:background}, we review some of the standard notation and terminology we use throughout the article. In Section~\ref{sec: Laurent polynomials}, we describe the additively reduced and additively atomic semidomains $S$ for which every Laurent polynomial $f \in S[x^{\pm 1}]$ that is not a monomial can be written as the sum of at most two multiplicative irreducibles, and we identify the ones that cannot be written as the sum of exactly two irreducibles. In particular, we show that for each $k \in \mathbb{N}$, every polynomial $f \in \mathbb{N}[x_1^{\pm 1}, \ldots, x_k^{\pm 1}]$ that is not a monomial can be written as the sum of two multiplicative irreducibles provided that $f(1, \ldots, 1) > 3$. In Section~\ref{sec:Lauren series}, we characterize the additively reduced and additively atomic semidomains $S$ for which every Laurent series $f \in S\llbracket x^{\pm 1} \rrbracket$ that is not a monomial can be written as the sum of at most three multiplicative irreducibles. For such an $S$, we show that every Laurent series that is not a polynomial can be written as the sum of at most three irreducibles in at least $2^{\aleph_0}$ ways.

\section{Preliminaries}
\label{sec:background}
We now review some of the standard notation and terminology that we will use later. For a comprehensive background on semiring theory, we recommend the monograph~\cite{JG1999}. Let $\mathbb{Q}$, $\mathbb{Z}$, $\mathbb{N}$, and $\mathbb{N}_0$ denote the set of rational numbers, integers, positive integers, and nonnegative integers, respectively. For $m,n \in \mathbb{N}_0$, let $\llbracket m,n \rrbracket \coloneqq \{k \in \mathbb{Z} \mid m \leq k \leq n\}$. Observe that if $m > n$, then $\llbracket m,n \rrbracket = \emptyset$. 

\subsection{Commutative monoids}

Throughout this paper, a \emph{monoid} is defined to be a semigroup with identity that is cancellative and commutative. Unless we specify otherwise, we will use multiplicative notation for monoids. For the rest of the section, let $M$ be a monoid. We use the notation $M^{\times}$ to denote the group of units (i.e., invertible elements) of $M$. We say that $M$ is \emph{reduced} provided that the group of units of $M$ is trivial. Given a subset $S$ of $M$, we let $\langle S \rangle$ denote the smallest submonoid of $M$ containing $S$. 

For elements $b,c \in M$, we say that $b$ \emph{divides} $c$ \emph{in} $M$ and write $b \mid_M c$ if there exists $b' \in M$ such that $c = b b'$. A submonoid $N$ of $M$ is \emph{divisor-closed} if for every $c \in N$ and $b \in M$ the relation $b \mid_M c$ implies that $b \in N$. Let $S$ be a nonempty subset of $M$. We use the term \emph{common divisor} of $S$ to refer to an element $d\in M$ that divides all elements of $S$. We call a common divisor $d$ of $S$ a \emph{greatest common divisor} if it is divisible by all other common divisors of $S$. We denote by $\gcd_M(S)$ the set consisting of all greatest common divisors of $S$ and drop the subscript when there is no risk of confusion.

An element $a \in M \setminus M^{\times}$ is called an \emph{atom} (or \emph{irreducible}) if for every $b, c \in M$ the equality $a = bc$ implies that either $b \in M^{\times}$ or $c \in M^{\times}$. We denote by the set of atoms of $M$ by $\mathcal{A}(M)$. We say that $M$ is \emph{atomic} if every element in $M \setminus M^{\times}$ can be written as a finite product of atoms. On the other hand, we say that $M$ is \emph{antimatter} if $\mathcal{A}(M) = \emptyset$. It is not hard to show that if $M$ is an atomic monoid that is antimatter, then $M$ is an abelian group.

\subsection{Semirings and Semidomains}
\smallskip

A \emph{semiring} $S$ is a (nonempty) set endowed with two binary operations denoted by `$\cdot$' and `$+$' and called \emph{multiplication} and \emph{addition}, respectively, such that the following conditions hold:
\begin{enumerate}
	\item $(S\setminus \{0\}, \cdot)$ is a commutative semigroup with an identity element denoted by $1$;
	
	\item $(S,+)$ is a monoid with its identity element denoted by $0$;
	
	\item $b \cdot (c+d)= b \cdot c + b \cdot d$ for all $b, c, d \in S$.
\end{enumerate}

\noindent We usually write $b c$ instead of $b \cdot c$ for elements $b,c$ in a semiring $S$. We would like to emphasize that a more general notion of a `semiring' does not usually assume the commutativity of the underlying multiplicative semigroup, but throughout this article we will assume that the multiplication operation is commutative. A subset $S'$ of a semiring $S$ is a \emph{subsemiring} of~$S$ if $(S',+)$ is a submonoid of $(S,+)$ that contains~$1$ and is closed under multiplication. Clearly, every subsemiring of $S$ is a semiring.

\begin{definition}
	A \emph{semidomain} is a subsemiring of an integral domain.
\end{definition}

Let $S$ be a semidomain. We say that $(S \setminus \{0\}, \cdot)$ is the \emph{multiplicative monoid} of $S$, and we denote it by $S^*$. Following standard notation from ring theory, we refer to the units of the multiplicative monoid $S^*$ simply as \emph{units}, and we let $S^\times$ denote the group of units of $S$. For $b,c \in S$ such that $b$ divides~$c$ in~$S^*$, we write $b \mid_S c$ (instead of $b \mid_{S^*} c$). Also, for a nonempty subset $B$ of $S$, we use $\gcd(B)$ to denote the set of greatest common divisors of $B$ in the monoid $S^*$. On the other hand, we denote the set of atoms of the multiplicative monoid $S^*$ by $\mathcal{A}(S)$ instead of $\mathcal{A}(S^*)$, while we denote the set of atoms of the additive monoid $(S,+)$ by $\mathcal{A}_{+}(S)$. 

A semidomain $S$ is \emph{additively reduced} if $0$ is the only invertible element of the monoid $(S,+)$. In this paper, we work with additively reduced semidomains $S$ for which $(S,+)$ is atomic and $\mathcal{A}_{+}(S) = S^{\times}$. In this context, we can write nonzero elements $s \in S^*$ as $s = u_s + v_s$, where $u_s \in S^{\times}$ and $v_s \in S$; observe that if $s \notin S^{\times}$ then $v_s \in S^*$. This class of semidomains includes familiar examples such as the semiring of nonnegative integers \( \mathbb{N}_0 \), which provides the natural framework for the Goldbach conjecture. Furthermore, this structural condition is preserved under passage to Laurent polynomial extensions: if a semidomain \( S \) satisfies the above condition then so does \( S[x^{\pm1}] \). We take advantage of this fact in Corollary~\ref{cor: Laurent polynomials with finitely many variables}. Finally, note that if $\mathcal{A}_{+}(S) \nsubseteq S^{\times}$, then for any $a \in \mathcal{A}_{+}(S) \setminus S^{\times}$ the polynomial $ax^2 + ax + a \in S[x^{\pm 1}]$ cannot be expressed as the sum of at most two multiplicative irreducibles.

\subsection{Laurent polynomials and Laurent series}

Let $R$ be an integral domain containing $S$ as a subsemiring. Then the \emph{semiring of polynomials $S[x]$ over~$S$} is a subsemiring of $R[x]$, and so $S[x]$ is also a semidomain. The elements of $S[x]$ are also polynomials in $R[x]$. As a result, all the standard terminology for polynomials can be applied to elements of $S[x]$, including \emph{degree} and \emph{leading coefficient}. Similarly, the \emph{semiring of Laurent series with finitely many negative terms $S\llbracket x^{\pm 1}\rrbracket$ over~$S$} is also a semidomain. Throughout the paper, when we refer to Laurent series, we refer only to Laurent series with finitely many negative terms. For the rest of this section, let $S$ be an additively reduced semidomain. The semiring of Laurent polynomials $S[x^{\pm 1}]$ over $S$ is embedded in $S\llbracket x^{\pm 1} \rrbracket$. In fact, since $S$ is additively reduced, the multiplicative monoid $S[x^{\pm 1}]^*$ is a divisor-closed submonoid of $S\llbracket x^{\pm 1} \rrbracket^*$. Consequently, the inclusion $\mathcal{A}(S[x^{\pm 1}]) \subseteq \mathcal{A}(S\llbracket x^{\pm 1} \rrbracket)$ holds. Moreover, it is not hard to see that
\[
	S\left\llbracket x^{\pm 1} \right\rrbracket^{\times} = S\left[x^{\pm 1}\right]^{\times} = \left\{sx^k \mid s \in S^{\times} \text{ and } k \in \mathbb{Z}\right\}.
\]
Finally, given a Laurent series $f = \sum_{i = 0}^{\infty} s_i x^{k_i} \in S\llbracket x^{\pm 1}\rrbracket$, we refer to the set $\{k_i \mid s_i \neq 0, \, i \in \mathbb{N}_0\}$ as the \emph{support} of $f$, which we denote by $\supp(f)$.

\section{Laurent polynomials as the sum of two irreducibles}
\label{sec: Laurent polynomials}

In this section, we describe the additively reduced and additively atomic semidomains $S$ for which every Laurent polynomial $f \in S[x^{\pm 1}]$ that is not a monomial can be written as the sum of at most two multiplicative irreducibles, and we identify the ones that cannot be written as the sum of exactly two irreducibles. In particular, we show that, for each $k \in \mathbb{N}$, every polynomial $f \in \mathbb{N}_0[x_1^{\pm 1}, \ldots, x_k^{\pm 1}]$ that is not a monomial can be written as the sum of two multiplicative irreducibles provided that $f(1, \ldots, 1) > 3$. This result generalizes~\cite[Theorem 1]{liao}.

Throughout this section, whenever we consider a polynomial expression $\sum_{i = 0}^{n} s_i x^{k_i} \in S[x^{\pm 1}]$, we assume that $k_0 > \cdots > k_n$ and $s_0, \ldots, s_n \in S^*$ unless we specify otherwise. Following~\cite{chapmanpolo}, we say that a nonzero polynomial $f \in S[x^{\pm 1}]$ is \emph{monolithic} if $f = gh$ implies that either $g$ or $h$ is a monomial in $S[x^{\pm 1}]$. Our next lemma explains the connection between monolithic and irreducible polynomials. Recall that for a nonempty subset $B$ of a semidomain $S$, $\gcd(B)$ is a (possibly empty) subset of $S$.

\begin{lemma} \label{lemma: monolithic and gcd 1 implies irreducibility}
	Let $f = \sum_{i = 0}^{n} s_ix^{k_i} \in S[x^{\pm 1}]$ such that $|\supp(f)| > 1$. Then $f$ is multiplicative irreducible in $S[x^{\pm 1}]$ if and only if $f$ is monolithic and $1 \in \gcd(s_0, \ldots, s_n)$. 
\end{lemma}

\begin{proof}
	The direct implication follows from the fact that $S[x^{\pm 1}]^{\times} = \{sx^k \mid s \in S^{\times} \text{ and } k \in \mathbb{Z}\}$. Suppose now that $f$ is monolithic and $1 \in \gcd(s_0, \ldots, s_n)$. Let $g, h \in S[x^{\pm 1}]$ such that $f = gh$. Since $f$ is monolithic, without loss of generality we may assume that $g$ is a monomial. Thus we can write $g = sx^k$ for some $s \in S$ and $k \in \mathbb{Z}$. Since $1 \in \gcd(s_0, \ldots, s_n)$, we have that $s \in S^{\times}$. Therefore $g$ is a unit of $S[x^{\pm 1}]$, which implies that $f$ is irreducible.
\end{proof}

Observe that being monolithic is a relaxation of the property of being irreducible. Next we introduce sufficient conditions under which a polynomial is monolithic. The following two lemmas will play an important role in the proof of the main result of this section. 

\begin{lemma} \label{lemma: all elements are in the lower half}
	Let $S$ be an additively reduced semidomain, and let $f = \sum_{i = 0}^{n} s_ix^{k_i} \in S[x^{\pm 1}]$. The following statements hold.
	\begin{enumerate}
		\item If $|\supp(f)| \geq 2$ and $k_1 < \frac{k_0 + k_n}{2}$ then $f$ is monolithic.
		\item If $|\supp(f)| > 3$ and $k_1 \leq \frac{k_0 + k_n}{2}$ then $f$ is monolithic.
	\end{enumerate}
\end{lemma}

\begin{proof}
Scaling $f$ by an appropriate power of $x$, we may assume that $k_n = 0$.  We argue by contradiction. Suppose that $f$ is not monolithic, that is, $f = gh$, where neither $g$ nor $h$ is a monomial in $S[x^{\pm 1}]$. 
Scaling by appropriate powers of $x$, we may also assume that $g$ and $h$ are elements of $S[x]$ satisfying $g(0)h(0) \neq 0$. 

Suppose that $k_1 < \frac{k_0}{2}$. Let $s$ and $s'$ be the leading coefficients of $g$ and $h$, respectively.  Since $S$ is additively reduced and $h(0) \neq 0$, we see that $\supp(g) \subseteq \supp(gh)$. Since $h$ is not a monomial, $\deg(gh) > \deg(g)$. We conclude that $\deg(g) \in \{k_1, \ldots, k_n\}$.  Consequently, $\deg(g) \le k_1 < \frac{k_0}{2}$. 
A similar argument implies that $\deg(h) < \frac{k_0}{2}$. Since $k_0 = \deg(f) = \deg(gh)$, this is a contradiction. We conclude that $f$ is monolithic.  This completes the proof that statement $(1)$ holds.
	
	We now prove statement $(2)$.  Without loss of generality, we may assume that $\ k_1 = \frac{k_0}{2}$ and $\deg(g) \ge \deg(h)$.
This implies that $\deg(g) \geq \frac{k_0}{2}$. Since $S$ is additively reduced, $h$ is not a monomial, and $h(0) \neq 0$, we must have $\deg(g) = \frac{k_0}{2}$. Hence $\deg(h) = \frac{k_0}{2}$. Since $|\supp(f)| > 3$, either $g$ or $h$ is not a binomial, which implies that $k_1 > \frac{k_0}{2}$.  This is a contradiction and so statement $(2)$ holds.
\end{proof}

\begin{lemma} \label{lemma: all elements are in the upper half}
	Let $S$ be an additively reduced semidomain, and let $f = \sum_{i = 0}^{n} s_ix^{k_i} \in S[x^{\pm 1}]$. The following statements hold.
	\begin{enumerate}
		\item If $|\supp(f)| \geq 2$ and $k_{n - 1} > \frac{k_0 + k_n}{2}$ then $f$ is monolithic.
		\item If $|\supp(f)| > 3$ and $k_{n - 1} \geq \frac{k_0 + k_n}{2}$ then $f$ is monolithic.
	\end{enumerate}
\end{lemma}

\begin{proof}
We follow the same strategy as in the proof of Lemma \ref{lemma: all elements are in the lower half}.  Scaling $f$ by an appropriate power of $x$, we may assume that $k_n = 0$.  
Suppose that $f = gh$ 
where neither $g$ nor $h$ is a monomial in $S[x^{\pm 1}]$.  Scaling by appropriate powers of $x$, we may also assume that $g$ and $h$ are elements of $S[x]$ satisfying $g(0)h(0) \neq 0$. 

Suppose that $k_{n - 1} > \frac{k_0}{2}$. Let $d =\deg(g)$ and $d'=\deg(h)$. Since $S$ is additively reduced and $g(0)h(0) \neq 0$, we see that $\min(d,d') \ge k_{n-1} > \frac{k_0}{2}$. This contradicts the fact that $d + d' = k_0$, and we conclude that statement (1) holds.
	
We now prove statement $(2)$. We may assume that $k_{n - 1} = \frac{k_0}{2}$. Observe that $\deg(g) \geq \frac{k_0}{2}$ since $S$ is additively reduced and $h(0) \neq 0$. Similarly, $\deg(h) \geq \frac{k_0}{2}$. Consequently, we must have $\deg(g) = \deg(h) = \frac{k_0}{2}$. Since $|\supp(f)| > 3$, either $g$ or $h$ is not a binomial, which implies that $k_{n - 1} < \frac{k_0}{2}$.  This is a contradiction, and therefore statement $(2)$ holds. 
\end{proof}

Now we are in a position to describe the additively reduced and additively atomic semidomains $S$ for which the Laurent polynomial extension $S[x^{\pm 1}]$ satisfies an analogue of the Goldbach conjecture. We will deal with binomials and trinomials separately, and then address the general case.

\begin{lemma} \label{lemma: binomials that can be written as the sum of two irreducibles}
	Let $S$ be an additively reduced and additively atomic semidomain for which $\mathcal{A}_{+}(S) = S^{\times}$\!. Suppose $f \in S[x^{\pm 1}]$ has $|\supp(f)| = 2$. The following statements hold.
	\begin{enumerate}
		\item If $f$ is not of the form $f = ax^{k_0} + bx^{k_1}$, where either $a \in S^{\times}\!$ or $b \in S^{\times}$ then $f$ is the sum of two multiplicative irreducibles. 
		\item If $f = ax^{k_0} + bx^{k_1}$, where either $a \in S^{\times}\!$ or $b \in S^{\times}$ then $f$ is irreducible.
	\end{enumerate} 
\end{lemma}

\begin{proof}
	We can write $f = ax^{k_0} + bx^{k_1}$ for some $a, b \in S^*$ and $k_0, k_1 \in \zz$. Since $(S,+)$ is atomic and $\mathcal{A}_{+}(S) = S^{\times}$, we can write $a = u_a + v_a$ and $b = u_b + v_b$, where $u_a, u_b \in S^{\times}$. If either $v_a = 0$ or $v_b = 0$ then $f$ is irreducible, and statement $(2)$ immediately follows. On the other hand, if $v_a v_b \neq 0$ then we can write
	\[
	f = \left[u_a x^{k_0} + v_b x^{k_1}\right] +  \left[v_a x^{k_0} + u_b x^{k_1}\right]\!,
	\]
	where the summands between brackets are irreducible by Lemma~\ref{lemma: monolithic and gcd 1 implies irreducibility}. Therefore statement $(1)$ holds. 
\end{proof}

\begin{lemma} \label{lemma: trinomials can be written as the sum of two irreducibles}
	Let $S$ be an additively reduced and additively atomic semidomain for which $\mathcal{A}_{+}(S) = S^{\times}$\!. Suppose $f \in S[x^{\pm 1}]$ has $|\supp(f)| = 3$. Then $f$ can be written as the sum of two multiplicative irreducibles unless $f$ is the sum of three units, in which case $f$ is irreducible. 
\end{lemma}

\begin{proof}
	Write $f = ax^{k_0} + bx^{k_1} + cx^{k_2}$ with $abc \neq 0$, and suppose that $f$ is not the sum of three units. We split our reasoning into the following two cases.
	
	\smallskip
	\noindent\textsc{Case 1:} $b \in S^{\times}\!$. Since $f$ is not the sum of three units, either $a \notin S^{\times}$ or $c \notin S^{\times}$. Suppose that $a \notin S^{\times}\!$. Since $(S,+)$ is atomic and $\mathcal{A}_{+}(S) = S^{\times}$, we can write $a = u_a + v_a$, where $u_a \in S^{\times}$ and $v_a \in S^*$. 
	Thus,
	\[
	f = \left[u_ax^{k_0} + cx^{k_2}\right] + \left[v_ax^{k_0} + bx^{k_1}\right].
	\]
	Since $b \in S^\times$, each binomial between brackets is irreducible by Lemma~\ref{lemma: monolithic and gcd 1 implies irreducibility}. The argument when $c \notin S^{\times}$ is identical.
	
	\smallskip
	\noindent\textsc{Case 2:} $b \notin S^{\times}$. We can write $b = u_b + v_b$ where $u_b \in S^{\times}$ and $v_b \in S^*$, and we can write $a = u_a + v_a$ where $u_a \in S^{\times}$ and $v_a \in S$. Thus,
	\begin{equation} \label{eq: second term is irreducible}
		f = \left[u_ax^{k_0} + v_bx^{k_1}\right] + \left[v_ax^{k_0} + u_bx^{k_1} + cx^{k_2}\right].
	\end{equation}
	The first summand in Equation~\eqref{eq: second term is irreducible} is irreducible as $v_b \in S^*$. If $v_a = 0$ then the second summand is also irreducible, which concludes the proof. Therefore, we suppose that $v_a \neq 0$.
	
    We argue by contradiction.  Suppose that $g = v_ax^{k_0} + u_bx^{k_1} + cx^{k_2}$ is reducible. Then $g$ is the product of two binomials. 
This implies that $u_b = s + s'$ for some $s,s' \in S^*$. Consequently, we have that $1 = u_b^{-1}s + u_b^{-1}s'$ which, in turn, implies that $\mathcal{A}_{+}(S) = \emptyset$. Indeed, every element $a \in S^*$ can be written as $a = au_b^{-1}s + au_b^{-1}s'$ for nonzero summands $au_b^{-1}s$ and $au_b^{-1}s'$. However, this contradicts the assumption that $1 \in \mathcal{A}_{+}(S)$. Therefore $g$ is irreducible, completing the proof in this case.

Arguing as in the proof of \textsc{Case 2}, if $f$ is the sum of three units then it is irreducible, completing the proof.
\end{proof}

Next we describe the additively reduced and additively atomic semidomains $S$ for which the Laurent polynomial extension $S[x^{\pm 1}]$ satisfies an analogue of the Goldbach conjecture. 

\begin{theorem} \label{theorem: our principal result}
	Let $S$ be an additively reduced and additively atomic semidomain. 
    \begin{enumerate}[leftmargin=*,label=(\Alph*)]
    \item    The following statements are equivalent:
	\begin{enumerate}[label=(\arabic*)]
		\item $\mathcal{A}_{+}(S) = S^{\times}$\!;
		
		\smallskip
		\item every $f \in S[x^{\pm 1}]$ with $|\supp(f)| > 1$ can be expressed as the sum of at most two multiplicative irreducibles;
		
		\smallskip
		\item there exists $k \in \nn$ such that every $f \in S[x^{\pm 1}]$ with $|\supp(f)| > 1$ can be expressed as the sum of at most $k$ multiplicative irreducibles.
	\end{enumerate}
	\smallskip
\item  Moreover, suppose any of the statements in (A) holds and $f \in S[x^{\pm 1}]$ does not have one of the following forms:
\begin{enumerate}
	\item[(a)] $f = ax^{k_0} + bx^{k_1}$\!, where either $a \in S^{\times}$\! or $b \in S^{\times}$\!;
	\item[(b)] $f= ax^{k_0} + bx^{k_1} + cx^{k_2}$, where $a,b,c \in S^{\times}$.
\end{enumerate}
Then $f$ is the sum of exactly two irreducible polynomials in $S[x^{\pm 1}]$.
\end{enumerate}
\end{theorem}

\begin{proof}
	$(2) \implies (3)\!:$ This is obvious. 
	
	$(3) \implies (1)\!:$ Since $S$ is additively reduced, if $a \in \mathcal{A}_{+}(S) \setminus S^{\times}$ then there is no constant $k \in \nn$ such that every element of $S[x^{\pm 1}]$ can be written as the sum of at most $k$ irreducibles. Indeed, for any $n \in \nn$, the polynomial $\sum_{i = 0}^n ax^{k_i}$ cannot be expressed as the sum of $n$ or fewer irreducibles. Consequently, the inclusion $\mathcal{A}_{+}(S) \subseteq S^{\times}$ holds. 
    
    We argue by contradiction. Let $u \in S^{\times}$\! and suppose that $u \notin \mathcal{A}_{+}(S)$. There exist $a,b\in S^*$ such that $u = a+b$.  This implies $1 = u^{-1}a + u^{-1}b$ which, in turn, implies that $\mathcal{A}_{+}(S) = \emptyset$. 
Since $(S,+)$ is atomic and antimatter, it must be an additive group. This contradicts the assumption that $S$ is an additively reduced semidomain. Therefore $\mathcal{A}_{+}(S) = S^{\times}$\!.
	
	$(1) \implies (2)\!:$ By virtue of Lemma~\ref{lemma: binomials that can be written as the sum of two irreducibles} and Lemma~\ref{lemma: trinomials can be written as the sum of two irreducibles}, we need only consider the case where $|\supp(f)| > 3$. Write $f = \sum_{i = 0}^n s_ix^{k_i}$, where $s_0, \ldots, s_n \in S^*$.  Scaling by an appropriate power of $x$, we may assume that 
$k_0 > \cdots > k_n = 0$. It is possible to write $s_0 = u_0 + v_0$ and $s_n = u_n + v_n$ with $u_0, u_n \in S^{\times}$\!. Let $f^* = \sum_{i = 0}^{n} s_i^* x^{k_i}$, where
	\[
		 s_i^* = \begin{cases} 
			v_{0} & i = 0,\\
			s_i & k_i > \frac{k_0}{2},\\
			u_n & i = n,\\ 
			0 & \text{otherwise}. 
		\end{cases}
	\] 
	Observe that $f^*$ is either a unit or an irreducible by Lemma~\ref{lemma: monolithic and gcd 1 implies irreducibility} and Lemma~\ref{lemma: all elements are in the upper half}. Similarly, let $f_{*} = \sum_{i = 0}^{n} t_{i}^* x^{k_i}$, where
	\[
	t_{i}^* = \begin{cases} 
		u_{0} & i = 0,\\
		s_i & k_i < \frac{k_0}{2},\\
		v_n & i = n,\\ 
		0 & \text{otherwise}. 
	\end{cases}
	\] 
	Again, $f_*$ is either a unit or an irreducible by Lemma~\ref{lemma: monolithic and gcd 1 implies irreducibility} and Lemma~\ref{lemma: all elements are in the lower half}. Note that since $|\supp(f)| > 3$, either $f^*$ or $f_*$ is not a unit. We split our reasoning into the following two cases.
	
	\smallskip
	\noindent\textsc{Case 1:} either $f^*$ or $f_*$ is a unit. Suppose that $f^*$ is a unit. In this case, we have that $s_0 \in S^{\times}$ and $k_1 \leq \frac{k_0}{2}$. Thus,
	\[
		f = \left[ s_1x^{k_1} + u_nx^{k_n} \right] + \left[s_0x^{k_0} + s_2x^{k_2} + \cdots + s_{n-1} x^{k_{n-1}} + v_nx^{k_n}\right].
	\]
	The first summand between brackets is clearly irreducible. The second summand is irreducible by Lemma~\ref{lemma: monolithic and gcd 1 implies irreducibility} and Lemma~\ref{lemma: all elements are in the lower half}. The case in which $f_*$ is a unit can be handled similarly. We leave its proof to the reader.
	
	\smallskip
	\noindent\textsc{Case 2:} neither $f^*$ nor $f_*$ is a unit. In this case, both $f^*$ and $f_*$ are irreducible. Note that $f = f^* + f_*$ unless there exists $j \in \llbracket 0,n \rrbracket$ satisfying that $k_j = \frac{k_0}{2}$. Consequently, let us assume that such an index $j \in \llbracket 0,n \rrbracket$ exists. If $|\supp(f^*)| \geq 3$ (resp., $|\supp(f_*)| \geq 3$) then $f^* + s_jx^{k_j}$ (resp.,  $f_* + s_jx^{k_j}$) is irreducible by Lemma~\ref{lemma: all elements are in the upper half} (resp., Lemma~\ref{lemma: all elements are in the lower half}), which concludes our argument. Hence, we may assume that $|\supp(f^*)| = |\supp(f_*)| = 2$. On the other hand, if $v_0 = 0$ (resp., $v_n = 0$) then $f^*+ s_jx^{k_j}$ (resp., $f_* + s_jx^{k_j}$) is irreducible by Lemma~\ref{lemma: monolithic and gcd 1 implies irreducibility} and Lemma~\ref{lemma: all elements are in the upper half} (resp., Lemma~\ref{lemma: all elements are in the lower half}), which concludes our argument. Therefore we need only consider the case where $v_0v_n \neq 0$. But this, along with the fact that $|\supp(f^*)| = |\supp(f_*)| = 2$, implies that $|\supp(f)| = 3$, which is a contradiction. 
	
	The last part of the theorem follows from Lemma~\ref{lemma: binomials that can be written as the sum of two irreducibles} and Lemma~\ref{lemma: trinomials can be written as the sum of two irreducibles}.
\end{proof}

\begin{cor} \label{cor: Goldbach conjecture for Laurent polynomials with positive integer}
	Every $f \in \mathbb{N}_0[x^{\pm 1}]$ can be written as the sum of two multiplicative irreducibles provided that $f(1) > 3$ and $|\supp(f)| > 1$.
\end{cor}

Corollary~\ref{cor: Goldbach conjecture for Laurent polynomials with positive integer} was first proved in \cite{liao} using information about the distribution of prime numbers in the positive integers. We have recovered this result without relying on specific information about prime numbers. In fact, we can now easily extend this result.

\begin{cor} \label{cor: Laurent polynomials with finitely many variables}
	Fix $k \in \mathbb{N}$. Every $f \in \mathbb{N}_0[x_1^{\pm 1}, \ldots, x_k^{\pm 1}]$ can be written as the sum of two multiplicative irreducibles provided that $f(1, \ldots, 1) > 3$ and $|\supp(f)| > 1$.
\end{cor}

\begin{proof}
	We proceed by induction on $k$. The result clearly holds for $k = 1$. Suppose that, for some $k \geq 1$, every polynomial $f \in \mathbb{N}_0[x_1^{\pm 1}, \ldots, x_k^{\pm 1}]$ can be written as the sum of two irreducibles provided that $f(1, \ldots, 1) > 3$ and $|\supp(f)| > 1$. Now set $S \coloneqq  \mathbb{N}_0[x_1^{\pm 1}, \ldots, x_k^{\pm 1}]$, and observe that the monoid $(S,+)$ is reduced and atomic. Moreover, it is not hard to see that
	\[
		\mathcal{A}_{+}(S) = S^{\times} = \left\{x_1^{n_1} \cdots x_k^{n_k} \mid n_i \in \mathbb{Z} \text{ for every } i \in \llbracket 1,k \rrbracket\right\}.
	\]  
We note that $g(x_1,\ldots, x_k) \in \mathcal{A}_{+}(S)$ if and only if $g(1,\ldots, 1) = 1$.

    By Theorem~\ref{theorem: our principal result}, we have that every $f \in S\left[x_{k + 1}^{\pm 1}\right]$ is the sum of exactly two multiplicative irreducibles unless $f$ has one of the following forms:
	\begin{enumerate}
		\item[(a)] $f = ax_{k + 1}^{t_0} + bx_{k + 1}^{t_1}$, where either $a \in S^{\times}$ or $b \in S^{\times}$;
		\item[(b)] $f= ax_{k + 1}^{t_0} + bx_{k + 1}^{t_1} + cx_{k + 1}^{t_2}$, where $a,b,c \in S^{\times}$.
	\end{enumerate}
	If $f= ax_{k + 1}^{t_0} + bx_{k + 1}^{t_1} + cx_{k + 1}^{t_2}$ or $f = ax_{k + 1}^{t_0} + bx_{k + 1}^{t_1}$, where $a,b,c \in S^{\times}$, then $f(1, \ldots, 1) \leq 3$. Therefore we need only consider the case where $f = ax_{k + 1}^{t_0} + bx_{k + 1}^{t_1}$ and either $a \notin S^{\times}$ or $b \notin S^{\times}$. Without loss of generality, assume that $a \in S^{\times}$ and $b \notin S^{\times}$. It is not hard to see that if $b$ is the sum of two units of $S$ then $f(1, \cdots, 1) = 3$. Then we may further assume that $b$ is not the sum of two units and we can write $b = u_b + u_b' + v_b$ with $u_b, u_b' \in S^{\times}$ and $v_b \in S^*$. In this case,
	\[
		f = \left[ax_{k + 1}^{t_0} + v_bx_{k + 1}^{t_1}\right] + \left[u_bx_{k + 1}^{t_1} + u_b'x_{k + 1}^{t_1}\right].
	\]
	It is clear that the summands between brackets are multiplicative irreducibles.
\end{proof}

Some of the auxiliary results established in the lead-up to Theorem~\ref{theorem: our principal result} have applications that extend beyond its proof, as the following example demonstrates.

\begin{example} \label{ex: polynomials with coefficients in the positive cone of Q satisfy Goldbach}
    The semidomain \( \mathbb{Q}_{\geq 0} \) is additively reduced but not additively atomic. Indeed, the set of additive atoms \( \mathcal{A}_{+}(\mathbb{Q}_{\geq 0}) \) is empty. As a result, we cannot apply Theorem~\ref{theorem: our principal result} directly to determine whether an analogue of the Goldbach conjecture holds for \( \mathbb{Q}_{\geq 0}[x^{\pm 1}] \). However, by combining Lemma~\ref{lemma: all elements are in the lower half}, Lemma~\ref{lemma: all elements are in the upper half}, and Theorem~\ref{theorem: our principal result}, we can still establish such a result. Showing that every binomial and trinomial in \( \mathbb{Q}_{\geq 0}[x^{\pm 1}] \) can be expressed as a sum of two multiplicative irreducibles is straightforward, so we leave this to the reader.

    Consider \( f \in \mathbb{Q}_{\geq 0}[x^{\pm 1}] \) with \( |\supp(f)| > 3 \) and choose \( N \in \mathbb{N} \) such that \( Nf \in \mathbb{N}_0[x^{\pm 1}] \). Theorem~\ref{theorem: our principal result} implies that we can write \( Nf = g + h \), where \( g \) and \( h \) are multiplicative irreducibles in \( \mathbb{N}_0[x^{\pm 1}] \). The argument from the proof of that theorem shows that each of \( g \) and \( h \) is either a binomial or satisfies one of the structural conditions in Lemma~\ref{lemma: all elements are in the lower half} or Lemma~\ref{lemma: all elements are in the upper half}.  If $g$ is a binomial, then $g/N$ is a multiplicative irreducible in $\mathbb{Q}_{\geq 0}[x^{\pm 1}]$. If $g$ satisfies one of the structural conditions in Lemma~\ref{lemma: all elements are in the lower half} or Lemma~\ref{lemma: all elements are in the upper half}, then because these lemmas hold for all additively reduced semidomains, $g/N$ is irreducible in $\mathbb{Q}_{\geq 0}[x^{\pm 1}]$.  The corresponding statements hold for $h$ and $h/N$. Therefore, every \( f \in \mathbb{Q}_{\geq 0}[x^{\pm 1}] \) with \( |\supp(f)| > 1 \) can be written as a sum of two multiplicative irreducibles.
\end{example}

For Laurent polynomials \( f, g, h \in S[x^{\pm 1}] \), we say that \( \{g, h\} \) is a \emph{Goldbach decomposition} of \( f \) if \( f = g + h \) and both \( g \) and \( h \) are multiplicative irreducibles in \( S[x^{\pm 1}] \). We conclude this section by pointing out that, for an additively reduced semidomain $S$ satisfying that $(S,+)$ is atomic and $\mathcal{A}_{+}(S) = S^{\times}$, the number of Goldbach decompositions of a Laurent polynomial in $S[x^{\pm 1}]$ depends on the semidomain $S$. While every $f \in \mathbb{N}_0[x^{\pm 1}]$ has finitely many Goldbach decompositions, there exist semidomains $S$ for which some polynomials in $S[x^{\pm 1}]$ have infinitely many Goldbach decompositions. Consider the following example.  

\begin{example}
	Consider the additive monoid $M = \langle (\frac{2}{3})^k \mid k \in \mathbb{Z}\rangle$, which is clearly reduced. It is known that $M$ is atomic and $\mathcal{A}(M) = \{(\frac{2}{3})^k \mid k \in \mathbb{Z}\}$ \cite[Proposition~3.5]{fG2018}. Observe that $M$ is, in fact, a semidomain as it is closed with respect to the usual multiplication of rational numbers. Denoting this semidomain by $S$, we have that $\mathcal{A}_{+}(S) = S^{\times}$. By Theorem~\ref{theorem: our principal result}, we know that a binomial $ax^{k_0} + bx^{k_1} \in S[x^{\pm 1}]$ can be written as the sum of two multiplicative irreducibles whenever $a,b \notin S^{\times}$. However, the polynomial $f = \frac{4}{3}x + 2$ has infinitely many Goldbach decompositions in $S[x^{\pm 1}]$. In fact, using the identity $2(\frac{2}{3})^n = 3(\frac{2}{3})^{n + 1}$, it is not hard to show that, for every $n \in \mathbb{N}$, we can write $\frac{4}{3} = (\frac{2}{3})^n + s_n$ for some $s_n \in S^*$. Thus, for every $n \in \mathbb{N}$, we have that
	\[
			f = \left[\left(\frac{2}{3}\right)^n\!\!x + 1\right] + \left[s_nx + 1\right],
	\]
	where each summand between brackets is irreducible.
\end{example}

\section{Laurent series as the sum of three irreducibles}
\label{sec:Lauren series}

Recall that for an additively reduced semidomain $S$, we denote the semidomain consisting of all Laurent series with coefficients in $S$ that have finitely many negative terms by $S\llbracket x^{\pm 1} \rrbracket$. In this section, we characterize the additively reduced semidomains $S$ such that $(S,+)$ is atomic and every Laurent series $f \in S\llbracket x^{\pm 1} \rrbracket$ that is not a monomial can be written as the sum of at most three multiplicative irreducibles. For such an $S$, we show that every Laurent series with coefficients in $S$ that is not a polynomial can be written as the sum of at most three irreducibles in at least $2^{\aleph_0}$ ways.

Throughout this section, whenever we consider a Laurent series expression $f = \sum_{i = 0}^{\infty} s_ix^{k_i} \in S\llbracket x^{\pm 1} \rrbracket$, we tacitly assume that $k_i < k_{i + 1}$ for every $i \in \mathbb{N}_0$ and, unless we specify otherwise, we also assume that $s_i \neq 0$ for every $i \in \mathbb{N}_0$. 

We start by extending the property of being monolithic to the context of Laurent series with coefficients in a semidomain. We say that a nonzero Laurent series $f \in S\llbracket x^{\pm 1} \rrbracket$ is \emph{monolithic} if $f = gh$ implies that either $g$ or $h$ is a monomial in $S\llbracket x^{\pm 1} \rrbracket$. We have the following analogue of Lemma~\ref{lemma: monolithic and gcd 1 implies irreducibility}.

\begin{lemma} \label{lemma: monolithic and gcd 1 implies irreducibility for series}
	Let $f = \sum_{i = 0}^{\infty} s_ix^{k_i} \in S\llbracket x^{\pm 1} \rrbracket$ with $s_i \in S$ for every $i \in \mathbb{N}_0$. Suppose that $|\supp(f)| > 1$. Then $f$ is multiplicative irreducible in $S\llbracket x^{\pm 1}\rrbracket$ if and only if $f$ is monolithic and $1 \in \gcd\left(\{s_i \mid i \in \mathbb{N}_0\}\right)$. 
\end{lemma}

\begin{proof}
The direct implication follows readily upon noticing
that $S\llbracket x^{\pm 1}\rrbracket^{\times} = \left\{sx^k \mid s \in S^{\times} \text{ and } k \in \mathbb{Z}\right\}$. The rest of the proof proceeds as in Lemma~\ref{lemma: monolithic and gcd 1 implies irreducibility}. We leave the details to the reader.
\end{proof}

Next we introduce sufficient conditions under which a Laurent series is monolithic. The following two lemmas will play an important role in the proof of the main result of this section. 

\begin{lemma} \label{lemma: hyper-monolithic Laurent series is irreducible}
	Suppose $f = \sum_{i = 0}^{\infty} s_ix^{k_i} \in S\llbracket x^{\pm 1} \rrbracket$ with $s_i \in S^*$ for every $i \in \mathbb{N}_0$ and $k_1 - k_0 < k_{i + 1} - k_i$ for every $i \in \mathbb{N}$. Then $f$ is monolithic.
\end{lemma}

\begin{proof}
	We argue by contradiction. Suppose that $f$ is not monolithic, that is, $f = gh$, where neither $g$ nor $h$ is a monomial in $S\llbracket x^{\pm 1}\rrbracket$. We can write $g =\sum_{i=0}^\infty d_ix^{t_i}$ and $h=\sum_{i=0}^{\infty}e_ix^{r_i}$, where $d_0, d_1, e_0, e_1 \in S^*$ and $d_i, e_i \in S$ for every $i \in \mathbb{N}$. It is possible that $g$ or $h$ could be a polynomial that is not a monomial. Note that $k_0 = t_0 + r_{0}$. Switching the roles of $g$ and $h$ if necessary, assume that $t_{1} -t_0 \leq r_{1} - r_{0}$. This implies that $k_{1} = t_{1} + r_{0}$. Thus,
	\[
	t_{1} - t_0 = (t_{1} + r_{0}) - (t_0 + r_{0}) = k_{1} - k_0.
	\] 
	Since $S$ is additively reduced and $k_1 < t_1 + r_1$, we see that $t_1 + r_1 = k_{j + 1}$ for some $j \geq 1$. Since $t_0 + r_1$ is in the support of $f$, we see that $k_j \geq t_0 + r_1$. For this $j$ we have
	\[
		k_{j + 1} - k_j \leq (t_1 + r_1) - (t_0 + r_1) = t_1 - t_0 = k_1 - k_0.
	\]
	This is a contradiction and we conclude that $f$ is monolithic.
\end{proof}

\begin{lemma} \label{lemma: eventually increasing sequence of exponents implies irreducibility}
	Let $f = \sum_{i = 0}^{\infty} s_ix^{k_i} \in S\llbracket x^{\pm 1} \rrbracket$ with $s_i \in S^*$ for every $i \in \mathbb{N}_0$. Suppose that for some $N \in \mathbb{N}$, the sequence $(k_{N + i + 1} - k_{N + i})_{i \in \mathbb{N}}$ is strictly increasing. Then $f$ is monolithic.
\end{lemma}

\begin{proof}
	We follow the same strategy as in the proof of Lemma~\ref{lemma: hyper-monolithic Laurent series is irreducible}. Suppose that $f = gh$, where neither $g$ nor $h$ is a monomial in $S\llbracket x^{\pm 1} \rrbracket$. We can write $g =\sum_{i=0}^\infty d_ix^{t_i}$ and $h=\sum_{i=0}^{\infty}e_ix^{r_i}$, where $d_0, d_1, e_0, e_1 \in S^*$ and $d_i, e_i \in S$ for every $i \in \mathbb{N}$. It is possible that $g$ or $h$ could be a polynomial that is not a monomial. If $g$ and $h$ are both polynomials then so is $f$, which is a contradiction. Switching the roles of $g$ and $h$ if necessary, we can assume that $h$ is not a polynomial. Therefore, we may assume that $e_i \in S^*$ for every $i \in \mathbb{N}_0$. 
	
	Let $N' \in \mathbb{N}$ such that $N' > N$ and $k_{N' + 1} - k_{N'} > t_1 - t_0$. Since $S$ is additively reduced, $r_{N' + 1} \in \supp(h)$, and $t_0, t_1 \in \supp(g)$, we have that $\{t_0 + r_{N' + 1}, t_1 + r_{N' + 1}\} \subseteq \supp(f)$. This implies $k_j = t_0 + r_{N' + 1}$ for some $j$, and it is clear that $k_{j + 1} \leq t_1 + r_{N' + 1}$. Therefore $k_{j + 1} - k_j \leq t_1 - t_0$, which is a contradiction. We conclude that $f$ is monolithic.
\end{proof}

Now we are in a position to prove the main result of this section.

\begin{theorem} \label{theorem: our second principal result}
	Let $S$ be an additively reduced and additively atomic semidomain. The following statements are equivalent:
	\begin{enumerate}
		\item $\mathcal{A}_{+}(S) = S^{\times}$\!;
		
		\smallskip
		\item every $f \in S\llbracket x^{\pm 1} \rrbracket$ with $|\supp(f)| > 1$ can be expressed as the sum of at most three multiplicative irreducibles;
		
		\smallskip
		\item there exists $k \in \nn$ such that every $f \in S\llbracket x^{\pm 1}\rrbracket$ with $|\supp(f)| > 1$ can be expressed as the sum of at most $k$ multiplicative irreducibles.
	\end{enumerate}
\end{theorem}

\begin{proof}
	$(2) \implies (3)\!:$ This is obvious. 
	
	$(3) \implies (1)\!:$ Since $S$ is additively reduced, if $a \in \mathcal{A}_{+}(S) \setminus S^{\times}$ then there is no constant $k \in \nn$ such that every element of $S\llbracket x^{\pm 1} \rrbracket$ can be written as the sum of at most $k$ irreducibles. Indeed, the Laurent series $\sum_{i = 0}^\infty ax^{i}$ cannot be expressed as the sum of $k$ multiplicative irreducibles for any $k \in \mathbb{N}$. Consequently, the inclusion $\mathcal{A}_{+}(S) \subseteq S^{\times}$ holds. 
    
    We argue by contradiction.  Let $u \in S^{\times}$\! and suppose that $u \notin \mathcal{A}_{+}(S)$. There exist $a,b \in S^*$ such that $u = a+b$.  This implies $1 = u^{-1}a + u^{-1}b$ which, in turn, implies that $\mathcal{A}_{+}(S) = \emptyset$. 
Since $(S,+)$ is atomic and antimatter, it must be an additive group. This contradicts the assumption that $S$ is an additively reduced semidomain. Therefore, $\mathcal{A}_{+}(S) = S^{\times}$\!.
	
	$(1) \implies (2)\!:$ Let $f \in S\llbracket x^{\pm 1} \rrbracket$ with $|\supp(f)| > 1$. Note that if a Laurent polynomial $g \in S[x^{\pm 1}]$ is irreducible then it is also irreducible when considered as an element of $S\llbracket x^{\pm 1} \rrbracket$. Therefore if $f$ is a polynomial then the result follows from Theorem~\ref{theorem: our principal result} and the fact that $\mathcal{A}(S[x^{\pm 1}]) \subseteq \mathcal{A}(S\llbracket x^{\pm 1} \rrbracket)$.
	
	For the rest of the proof, suppose that $f$ is not a polynomial. Then we can write $f = \sum_{i = 0}^{\infty} s_ix^{k_i}$\!, where $s_i \in S^*$ for every $i \in \mathbb{N}_0$. Set $\Delta \coloneqq \min\left\{k_{i + 1} - k_i\right\}_{i \in \mathbb{N}_0}$ and let $J = \{j \in \mathbb{N}_0 \mid k_{j + 1} - k_j = \Delta\}$. We consider two cases.
	
	\smallskip
	\noindent\textsc{Case 1:} $J$ is a finite set. Let $\alpha$ be the smallest index in $J$. It is not hard to see that we can write $f = g + h$, where $g =\sum_{i=0}^\infty d_ix^{t_i}$, $h=\sum_{i=0}^{\infty}e_ix^{r_i}$, and the following conditions hold:
	\begin{enumerate}
		\item[(a)] $d_i, e_i \in S^*$ for every $i \in \mathbb{N}_0$;
		\smallskip
        \item[(b)] $d_0 = s_{\alpha}$, \quad $t_0 = k_{\alpha}$, \quad $d_1 = s_{\alpha + 1}$, \quad and \quad $t_1 = k_{\alpha + 1}$;
		\smallskip
        \item[(c)] $t_1 - t_0 < t_{i + 1} - t_i$ for every $i \in \mathbb{N}$;
		\smallskip
        \item[(d)] there exists $N \in \mathbb{N}$ such that the sequence $(r_{N + i + 1} - r_{N + i})_{i \in \mathbb{N}}$ is strictly increasing.
	\end{enumerate}
	Since $(S,+)$ is atomic and $\mathcal{A}_{+}(S) = S^{\times}$, we can write $d_2 = u_2 + v_2$ and $d_3 = u_3 + v_3$, where $u_2, u_3 \in S^{\times}$\!. Observe that 
    \[
    f = \left[g - v_2x^{t_{2}} - u_3x^{t_{3}}\right] + \left[h + v_2x^{t_{2}} + u_3x^{t_{3}}\right].
    \]
    Note that the first summand between brackets is irreducible by Lemma~\ref{lemma: monolithic and gcd 1 implies irreducibility for series} and Lemma~\ref{lemma: hyper-monolithic Laurent series is irreducible}. On the other hand, the second summand between brackets is irreducible by Lemma~\ref{lemma: monolithic and gcd 1 implies irreducibility for series} and Lemma~\ref{lemma: eventually increasing sequence of exponents implies irreducibility}.
	
	\smallskip
	\noindent\textsc{Case 2  :} $J$ is an infinite set. In this case, it is not hard to see that we can write $f = g + h$, where $g =\sum_{i=0}^\infty d_ix^{t_i}$, $h=\sum_{i=0}^{\infty}e_ix^{r_i}$, and the following conditions hold:
	\begin{enumerate}
		\item[(a)] $d_i, e_i \in S^*$ for all $i \in \mathbb{N}_0$;
		\smallskip
        \item[(b)] $d_0, d_2 \in S^{\times}$ and $t_3 - t_2 = t_1 - t_0 = \Delta$;
		\smallskip
        \item[(c)] $e_\ell \in S^{\times}$ for some $\ell \in \mathbb{N}_0$;
		\smallskip
        \item [(d)] there exists $N \in \mathbb{N}$ such that the sequence $(r_{N + i + 1} - r_{N + i})_{i \in \mathbb{N}}$ is strictly increasing.
	\end{enumerate}
	Thus,
	\[
		f = \left[d_0x^{t_0} + d_1x^{t_1} + \sum_{\substack{i = 4\\i \text{ is even}}}^{\infty}d_ix^{t_i} \right] + \left[d_2x^{t_2} + d_3x^{t_3} + \sum_{\substack{i = 4\\i \text{ is odd}}}^{\infty}d_ix^{t_i}\right] + h.
	\]
	The first and second summand between brackets are irreducible by Lemma~\ref{lemma: monolithic and gcd 1 implies irreducibility for series} and Lemma~\ref{lemma: hyper-monolithic Laurent series is irreducible} and $h$ is irreducible by Lemma~\ref{lemma: monolithic and gcd 1 implies irreducibility for series} and Lemma~\ref{lemma: eventually increasing sequence of exponents implies irreducibility}.
\end{proof}

\begin{cor} \label{cor: Goldbach conjecture for Laurent series with positive integer}
	Every $f \in \mathbb{N}_0\llbracket x^{\pm 1} \rrbracket$ can be written as the sum of at most three multiplicative irreducibles provided that $f(1) > 3$ and $|\supp(f)| > 1$.
\end{cor}

Corollary~\ref{cor: Goldbach conjecture for Laurent series with positive integer} implies that most $f \in \mathbb{N}_0\llbracket x^{\pm 1} \rrbracket$ can be written as the sum of at most three multiplicative irreducibles. We have already established that every polynomial with coefficients in \(\mathbb{N}_0\) can be written as the sum of at most two multiplicative irreducibles and can show the corresponding statement for many power series.

\begin{example}
    The power series \( f = \sum_{i = 0}^{\infty} x^i \in \mathbb{N}_0\llbracket x^{\pm 1} \rrbracket \) can be written as
    \[
        f = \left[1 + x + \sum_{i = 2}^{\infty} x^{2i}\right] + \left[x^2 + x^3 + \sum_{i = 2}^{\infty} x^{2i + 1}\right],
    \]
    where each summand between brackets is irreducible by Lemma~\ref{lemma: monolithic and gcd 1 implies irreducibility for series} and Lemma~\ref{lemma: hyper-monolithic Laurent series is irreducible}. Thus, \( f \) can be expressed as the sum of two multiplicative irreducibles.

    A similar decomposition applies to any power series \( g = \sum_{i = 0}^{\infty} s_i x^{k_i} \in \mathbb{N}_0\llbracket x^{\pm 1} \rrbracket \) whose support satisfies the spacing condition $k_1 - k_0 = k_2 - k_1 \leq k_{i+1} - k_i$ 
    for all \( i \in \mathbb{N}_0 \). In such cases, the structure of the support ensures that \( g \) can also be written as the sum of two multiplicative irreducibles.
\end{example}

The previous example illustrates that even though Theorem~\ref{theorem: our second principal result} provides a general upper bound of three summands, many series, especially those with well-structured supports, require only two. This suggests that the behavior we observed for polynomials may extend more broadly. Motivated by this observation, we propose a refinement of Theorem~\ref{theorem: our second principal result}.

\begin{conj}
	Let $S$ be an additively reduced and additively atomic semidomain. The following statements are equivalent:
	\begin{enumerate}
		\item $\mathcal{A}_{+}(S) = S^{\times}$\!;
		
		\smallskip
		\item every $f \in S\llbracket x^{\pm 1} \rrbracket$ with $|\supp(f)| > 1$ can be expressed as the sum of at most two multiplicative irreducibles;
		
		\smallskip
		\item there exists $k \in \nn$ such that every $f \in S\llbracket x^{\pm 1}\rrbracket$ with $|\supp(f)| > 1$ can be expressed as the sum of at most $k$ multiplicative irreducibles.
	\end{enumerate}
\end{conj}

In the previous section, we pointed out that, in an additively reduced and additively atomic semidomain $S$ for which $\mathcal{A}_{+}(S) = S^{\times}$\!, the number of Goldbach decompositions of a Laurent polynomial in $S[x^{\pm 1}]$ depends on the semidomain $S$. Next we show that a Laurent series $f \in S\llbracket x^{\pm 1} \rrbracket$ that is not a polynomial has uncountably many representations as the sum of at most three multiplicative irreducibles.

For a nonzero Laurent series $f \in S\llbracket x^{\pm 1} \rrbracket^*$, we denote by $\mathfrak{R}(f)$ the set consisting of all unordered triples $(g_1, g_2, g_3)$ satisfying that $f = g_1 + g_2 + g_3$ with $g_1, g_2, g_3 \in \{0\} \cup \mathcal{A}(S\llbracket x^{\pm 1} \rrbracket)$. The elements of $\mathfrak{R}(f)$ represent the different ways in which we can write $f$ as the sum of at most three irreducibles. 

\begin{prop}
		Let $S$ be an additively reduced and additively atomic semidomain for which $\mathcal{A}_{+}(S) = S^{\times}$\!. Suppose that $f \in S\llbracket x^{\pm 1} \rrbracket$ is not a polynomial. Then $|\mathfrak{R}(f)| \geq 2^{\aleph_0}\!$.
\end{prop}

\begin{proof}
	Write $f = \sum_{i = 0}^{\infty} s_ix^{k_i}$, where $s_i \in S^*$ for every $i \in \mathbb{N}_0$. We prove this statement by showing that there is an injective function from the set of infinite subsets of $\mathbb{N}$ to $\mathfrak{R}(f)$. We follow the proof of Theorem~\ref{theorem: our second principal result} closely. Set $\Delta \coloneqq \min\left\{k_{i + 1} - k_i\right\}_{i \in \mathbb{N}_0}$ and let $J = \{j \in \mathbb{N}_0 \mid k_{j + 1} - k_j = \Delta\}$. We consider two cases.
	
	\smallskip
	\noindent\textsc{Case 1:} $J$ is a finite set. Let $\alpha$ and $\beta$ be the smallest and biggest index in $J$, respectively. We already established that we can write $f = g + h$, where $g =\sum_{i=0}^\infty d_ix^{t_i}$, $h=\sum_{i=0}^{\infty}e_ix^{r_i}$, and the following conditions hold:
	\begin{enumerate}
		\item[(a)] $d_i, e_i \in S^*$ for every $i \in \mathbb{N}_0$;
		\item[(b)] $d_0 = s_{\alpha}$, \quad $t_0 = k_{\alpha}$, \quad $d_1 = s_{\alpha + 1}$, \quad and \quad $t_1 = k_{\alpha + 1}$;
		\item[(c)] $t_1 - t_0 < t_{i + 1} - t_i$ for every $i \in \mathbb{N}$;
		\item[(d)] there exists $N \in \mathbb{N}$ such that the sequence $(r_{N + i + 1} - r_{N + i})_{i \in \mathbb{N}}$ is strictly increasing.
	\end{enumerate}
	It is not hard to see that we can strengthen condition $(d)$ as follows:
	\begin{enumerate}
		\item[(d')] there exists $N \in \mathbb{N}$ such that $r_{N + i + 1} - r_{N + 1} < r_{N + i + 2} - r_{N + i + 1}$ for every $i \in \mathbb{N}$.
	\end{enumerate}
	It is clear that such an $N$ exists for which $r_{N + 1} > \max(t_3, k_{\beta + 1})$.  Choose such a value of $N$.
	
	Let $K$ be an arbitrary infinite subset of the positive integers. Since $(S,+)$ is atomic and $\mathcal{A}_{+}(S) = S^{\times}\!$, we can write $d_2 = u_2 + v_2$ and $d_3 = u_3 + v_3$, where $u_2, u_3 \in S^{\times}$\!. Observe that
    \begin{equation*}
        \begin{split}
        f = \ & \left[ g - v_2x^{t_2} - u_3x^{t_3} + \sum_{i \,\in\, \mathbb{N} \setminus K} e_{N+i} \, x^{r_{N+i}} \right] \\
        & + \left[ h + v_2x^{t_2} + u_3x^{t_3} - \sum_{i \,\in\, \mathbb{N} \setminus K} e_{N+i} \, x^{r_{N+i}} \right].
        \end{split}
    \end{equation*}
	The first summand between brackets is irreducible by Lemma~\ref{lemma: monolithic and gcd 1 implies irreducibility for series} and Lemma~\ref{lemma: hyper-monolithic Laurent series is irreducible}. The second summand between brackets is irreducible by Lemma~\ref{lemma: monolithic and gcd 1 implies irreducibility for series} and Lemma~\ref{lemma: eventually increasing sequence of exponents implies irreducibility}. Since distinct infinite subsets of the positive integers are associated to distinct unordered triples of $\mathcal{R}(f)$, our result follows.
	
	\smallskip
	\noindent\textsc{Case 2  :} $J$ is an infinite set. We already established that we can write $f = g + h$, where $g =\sum_{i=0}^\infty d_ix^{t_i}$, $h=\sum_{i=0}^{\infty}e_ix^{r_i}$, and the following conditions hold:
	\begin{enumerate}
	\item[(a)] $d_i, e_i \in S^*$ for all $i \in \mathbb{N}_0$;
    \smallskip
	\item[(b)] $d_0, d_2 \in S^{\times}$ and $t_3 - t_2 = t_1 - t_0 = \Delta$;
	\smallskip
    \item[(c)] $e_\ell \in S^{\times}$ for some $\ell \in \mathbb{N}_0$;
	\smallskip
    \item [(d)] there exists $N \in \mathbb{N}$ such that the sequence $(r_{N + i + 1} - r_{N + i})_{i \in \mathbb{N}}$ is strictly increasing.
	\end{enumerate}
	As in \textsc{Case 1}, we can strengthen condition $(d)$ as follows:
	\begin{enumerate}
		\item[(d')] there exists $N \in \mathbb{N}$ such that $r_{N + i + 1} - r_{N + 1} < r_{N + i + 2} - r_{N + i + 1}$ for every $i \in \mathbb{N}$.
	\end{enumerate}
	It is clear that such an $N$ exists for which $r_{N + 1} > \max(t_3, r_{\ell})$. Choose such a value of $N$.
	
	Let $K$ be an arbitrary infinite subset of the positive integers. Let $h^* = h - \sum_{i \in \mathbb{N} \setminus K} e_{N + i}x^{r_{N+i}}$ and $g^* = g + \sum_{i \in \mathbb{N} \setminus K} e_{N + i}x^{r_{N+i}}$. Write $g^* =\sum_{i=0}^\infty b_ix^{\ell_i}$ where $\ell_0 < \ell_1 < \cdots$ and each $b_i \in S^*$. By Lemma~\ref{lemma: monolithic and gcd 1 implies irreducibility for series} and Lemma~\ref{lemma: eventually increasing sequence of exponents implies irreducibility},  $h^*$ is irreducible. It is easy to see that the following conditions hold:
	\begin{enumerate}
		\item[($a^*$)] $b_0, b_2 \in S^{\times}$ and $\ell_3 - \ell_2 = \ell_1 - \ell_0 = \Delta$;
        \smallskip
		\item[($b^*$)] $b_i \in S^*$ for all $i \in \mathbb{N}_0$.
	\end{enumerate}
	We see that
	\[
	f = \left[b_0x^{\ell_0} + b_1x^{\ell_1} + \sum_{\substack{i = 4\\i \text{ is even}}}^{\infty}b_ix^{\ell_i} \right] + \left[b_2x^{\ell_2} + b_3x^{\ell_3} + \sum_{\substack{i = 4\\i \text{ is odd}}}^{\infty}b_ix^{\ell_i}\right] + h^*.
	\]
	The first and second summand between brackets are irreducible by Lemma~\ref{lemma: monolithic and gcd 1 implies irreducibility for series} and Lemma~\ref{lemma: hyper-monolithic Laurent series is irreducible}. Since distinct infinite subsets of the positive integers are associated to distinct unordered triples of $\mathcal{R}(f)$, our result follows.
\end{proof}

\section*{Acknowledgments}

The authors thank the anonymous referee whose suggestions helped improve this paper. The first author was supported by NSF grant DMS 2154223. The second author was supported by a University of California President's Postdoctoral Fellowship.


\begin{thebibliography}{20}

    
		
	\bibitem{bender} {\scshape Bender, A.O.} Representing an element in $\mathbb{F}_q[t]$ as the sum of two irreducibles. {\em Mathematika} {\bf 60} (2014) 166--182. 
	
	\bibitem{CF19} {\scshape Campanini, F.; Facchini, A.} Factorizations of polynomials with integral non-negative coefficients. {\em Semigroup Forum } {\bf 99} (2019) 317--332. 
		
	\bibitem{car} {\scshape Car, M.; Gallardo, L.H.} Representation of a polynomial as the sum of an irreducible polynomial and a square-free polynomial. {\em Acta Arith.} {\bf 197} (2021) 293--309. 
	
	\bibitem{chapmanpolo} {\scshape Chapman, S.T.; Polo, H.} Arithmetic of additively reduced monoid semidomains. {\em Semigroup Forum} {\bf 107} (2023) 40--59. 

	\bibitem{effinger} {\scshape Effinger, G.W.; Hayes, D.R.} A complete solution to the polynomial $3$-primes problem. {\em Bull. Amer. Math. Soc.} {\bf 24} (1991) 363--369. 
		
	\bibitem{goel} {\scshape Fox, H.; Goel, A.; Liao, S.} Arithmetic of semisubtractive semidomains. {\em J. Algebra Appl.} {\bf 24} (2025) 2550163. 

	\bibitem{JG1999} {\scshape Golan, J.S.} Semirings and their Applications. {\em Kluwer Academic Publishers}, 1999. 
	
	\bibitem{fG2018} {\scshape Gotti, F.} Puiseux monoids and transfer homomorphisms. {\em J. Algebra} {\bf 516} (2018) 95--114. 

	\bibitem{polo} {\scshape Gotti, F.; Polo, H.} On the arithmetic of polynomial semidomains. {\em Forum Math.} {\bf 35} (2023) 1179--1197. 

	\bibitem{hayes} {\scshape Hayes, D.R.} A Goldbach theorem for polynomials with integral coefficients. {\em Amer. Math. Monthly} {\bf 72} (1965) 45--46.

	\bibitem{kozek} {\scshape Kozek, M.} An asymptotic formula for Goldbach's conjecture with monic polynomials in $\mathbb{Z}[x]$. {\em Amer. Math. Monthly} {\bf 117} (2010) 365--369.

	\bibitem{liao} {\scshape Liao, S.; Polo, H.} A Goldbach theorem for Laurent polynomials with positive integer coefficients. {\em Amer. Math. Monthly} {\bf 131} (2024), no. 8, 704--711. 
	
	\bibitem{paran} {\scshape Paran, E.} Twin-prime and Goldbach theorems for $\mathbb{Z}[[x]]$. {\em J. Number Theory} {\bf 213} (2020) 453--461. 
		
	\bibitem{pollack} {\scshape Pollack, P.} On polynomial rings with a Goldbach property. {\em Amer. Math. Monthly} {\bf 118} (2011) 71--77. 	

	\bibitem{saidak} {\scshape Saidak, F.} On Goldbach's conjecture for integer polynomials. {\em Amer. Math. Monthly} {\bf 113} (2006) 541--545. 
\end{thebibliography}
\end{document}